\documentclass{amsart}

\usepackage{amssymb,verbatim}
\usepackage[all]{xy}

\numberwithin{equation}{section}

\theoremstyle{plain}
\newtheorem{theorem}{Theorem}[section]
\newtheorem{corollary}[theorem]{Corollary}
\newtheorem{lemma}[theorem]{Lemma}
\newtheorem{proposition}[theorem]{Proposition}
\theoremstyle{remark}
\newtheorem{remark}[theorem]{Remark}
\theoremstyle{definition}
\newtheorem{definition}[theorem]{Definition}
\newtheorem{notation}[theorem]{Notation}

\makeatletter

\def\R{\mathbb{R}}
\def\C{\mathbb{C}}
\def\Z{\mathbb{Z}}

\newcommand{\lie}[1]{\mathfrak{#1}}
\newcommand{\lier}[1]{\mathfrak{#1}_{0}}
\newcommand{\roots}[2]{\Delta({#1}, {#2})}
\newcommand{\proots}[2]{\Delta^{+}({#1}, {#2})}

\newcommand{\stWhmg}[2]{I_{\eta,{#1},{#2}}}

\newcommand{\stWhLmg}[1]{\stWhmg{\Lambda}{#1}}

\def\stWhLsmg{\stWhLmg{\sigma}}

\def\stWhLgmg{\stWhLmg{\gamma}}

\def\gK{\lie{g},K}
\def\brgK{(\gK)}
\def\rhom{\rho_{\lie{m}}}

\def\Hom{\mathrm{Hom}}

\def\Ext{\mathrm{Ext}}

\def\tr{\mathrm{tr}}

\def\Ind{\mathrm{Ind}}
\def\Ad{\mathrm{Ad}}
\def\ad{\mathrm{ad}}
\def\pr{\mathrm{pr}}
\def\Dim{\mathrm{Dim}}

\def\Wh{\mathrm{Wh}}
\def\sgn{\mathrm{sgn}}

\def\d{\partial}
\newcommand{\vect}[1]{\boldsymbol{#1}}
\def\I{\sqrt{-1}}

\allowdisplaybreaks[3]

\title[Standard Whittaker $\brgK$-modules]
{Socle filtrations of the standard Whittaker
  $\brgK$-modules of $Spin(r,1)$}
\author{Kenji Taniguchi}
\address{
Department of Physics and Mathematics, 
Aoyama Gakuin University, 
5-10-1, Fuchinobe, Chuo-ku, Sagamihara, Kanagawa 252-5258, Japan. }
\email{taniken@gem.aoyama.ac.jp}

\begin{document}

\begin{abstract}
Studied are the composition series of the standard Whittaker
$\brgK$-modules. 
For a generic infinitesimal character, the structures of
these modules are completely understood, but if the infinitesimal
character is integral, 
then there are not so many cases in which the structures of them are
known. 
In this paper, as an example of the integral case, 
we determine the socle filtrations of the standard Whittaker
$\brgK$-modules when $G$ is the group $Spin(r,1)$ and the infinitesimal
character is regular integral. 
\end{abstract}

\maketitle

\section{Introduction}
\label{section:introduction}

This paper is a continuation of the paper \cite{T2} in which the
author defined and examined the standard Whittaker $\brgK$-modules for
real reductive linear Lie groups. 

Let us review the definition of these modules. 
Let $G$ be a real reductive linear Lie group in the sense of \cite{V2} 
and $G=KAN$ be an Iwasawa decomposition of it. 
Let $\eta: N \longrightarrow \C^{\times}$ be a unitary character of $N$ 
and denote the differential representation $\lier{n} \to \I \R$ 
of it by the same letter $\eta$. 
We assume $\eta$ is non-degenerate, i.e. it is
non-trivial on every root space corresponding to a simple root of
$\Delta^{+}(\lier{g}, \lier{a})$. 
Define 
\begin{equation}\label{eq:space of Whittaker functions}
C^{\infty}(G/N; \eta) 
:= \{f: G \overset{C^{\infty}}{\longrightarrow} \C \,|\, 
f(gn) = \eta(n)^{-1} f(g), \enskip g \in G, n \in N\}
\end{equation}
and call it the space of Whittaker functions on $G$. 
This is a representation space of
$G$ by the left translation, which is denoted by $L$. 
Let $C^{\infty}(G/N;\eta)_{K}$ be the subspace of
$C^{\infty}(G/N;\eta)$ consisting of $K$-finite vectors. 
Let, as usual, $M$ be the centralizer of $A$ in $K$, and let 
$M^{\eta}$ 
be the stabilizer of $\eta$ in $M$. 
This subgroup acts naturally on $C^{\infty}(G/N; \eta)_{K}$ by the
right translation. 
Consider the subspace of $C^{\infty}(G/N; \eta)_{K}$ consisting of
those functions $f$ which satisfy the following conditions: 
\begin{enumerate}
\item
$f$ is a joint eigenfunction of $Z(\lie{g})$ (the center of
  the universal enveloping algebra $U(\lie{g})$) 
with eigenvalue 
$\chi_{\Lambda}$: 
$L(z) f = \chi_{\Lambda}(z) f$, $z \in Z(\lie{g})$. 
\item
For an irreducible representation $(\sigma, V_{\sigma}^{M^{\eta}})$ 
of $M^{\eta}$, 
$f$ is in the $\sigma^{\ast}$-isotypic subspace 
($\sigma^{\ast}$ is the dual of $\sigma$) 
with respect to the right action of $M^{\eta}$. 
\item
$f$ grows moderately at infinity (\cite{W}).  
\end{enumerate}
The space of functions which satisfy the above conditions
(1)--(3) is isomorphic to 
\begin{align}
\stWhLsmg 
:= 
\{f: G \overset{C^{\infty}}{\longrightarrow}
V_{\sigma}^{M^{\eta}} 
|\, 
&
f(gmn) = \eta(n)^{-1} \sigma(m)^{-1} f(g), 
\enskip g \in G, m \in M^{\eta}, n \in N; 
\notag\\
& L(z) f = \chi_{\Lambda}(z) f, \enskip z \in Z(\mathfrak{g}); 
\enskip 
\mbox{left $K$-finite};
\notag\\
& f \mbox{ grows moderately at infinity} 
\}.
\notag
\end{align}
We call this space the {\it standard Whittaker $\brgK$-module}. 

For ``generic'' infinitesimal character $\Lambda$, the structure of
$\stWhLsmg$ is completely determined in \cite{T2}. 
On the other hand, if the infinitesimal character $\Lambda$ is integral, 
its structure is not known except for the case $G = SL(2,\R)$ or
$U(n,1)$, $n \geq 2$. 
In this paper, we examine the case $G=Spin(r,1)$, $r \geq 3$, 
so that it will become a good example for the study of the case of
other general groups. 

The main results of this paper are as follows. 
For the notation 
of
irreducible modules and the diagrammatic expression of the composition
series, see \S~\ref{subsection:classification of (g,K)-modules}. 
\begin{theorem}
\label{theorem:main, 2n}
Suppose $G = Spin(2n,1)$ and the infinitesimal character 
$\Lambda=(\Lambda_{1}, \dots, \Lambda_{n})$ 
($\Lambda_{1} > \Lambda_{2} > \dots > \Lambda_{n} > 0$) is
regular integral. 
Let $\sigma$ be an irreducible representation of 
$M^{\eta} \simeq Spin(2n-2)$. 
\begin{enumerate}
\item
$\stWhLsmg$ is not zero if and only if the highest weight 
$\gamma = (\gamma_{1}, \dots, \gamma_{n-1})$ of 
$\sigma$ satisfies one of the following conditions: 
\begin{align}
&
\begin{cases}
\Lambda_{p}-n+p+1/2
\geq \gamma_{p} \geq 
\Lambda_{p+1}-n+p+3/2,
\quad 
p=1,\dots,n-1, 
\\
\Lambda_{n-1}-1/2
\geq \gamma_{n-1} \geq 
\Lambda_{n}+1/2, 
\end{cases}\label{eq:condition for sigma,2n,0} 
\\
&
\begin{cases}
\Lambda_{p}-n+p+1/2
\geq \gamma_{p} \geq 
\Lambda_{p+1}-n+p+3/2,
\quad 
p=1,\dots,n-1, 
\\
-\Lambda_{n}-1/2
\geq \gamma_{n-1} \geq 
-\Lambda_{n-1}+1/2, 
\end{cases}\label{eq:condition for sigma,2n,1} 
\\
&
\begin{cases}
i \in \{2,\dots,n\},
\\
\Lambda_{p}-n+p+1/2
\geq \gamma_{p} \geq 
\Lambda_{p+1}-n+p+3/2, 
\quad 
p=1,\dots,i-2, 
\\
\Lambda_{p+1}-n+p+1/2
\geq \gamma_{p} \geq 
\Lambda_{p+2}-n+p+3/2, 
\quad 
p=i-1,\dots,n-2, 
\\
\Lambda_{n}-1/2
\geq |\gamma_{n-1}|. 
\end{cases}\label{eq:condition for sigma,2n,0i}
\end{align}
\item
If \eqref{eq:condition for sigma,2n,0} 
(resp. \eqref{eq:condition for sigma,2n,1}) is satisfied, 
then 
$\displaystyle 
I_{\eta,\Lambda,\sigma}
\simeq 
\begin{xy}
(0,5)*{\pi_{1}}="A_{1}",
(0,0)*{\overline{\pi}_{0,n}}="A_{2}",
(0,-5)*{\pi_{0}}="A_{3}",
\ar "A_{1}";"A_{2}"
\ar "A_{2}";"A_{3}"
\end{xy} 
\left(
\mbox{resp. } 
\begin{xy}
(0,5)*{\pi_{0}}="A_{1}",
(0,0)*{\overline{\pi}_{0,n}}="A_{2}",
(0,-5)*{\pi_{1}}="A_{3}",
\ar "A_{1}";"A_{2}"
\ar "A_{2}";"A_{3}"
\end{xy}
\right)$. 

\vspace{2mm}

\item
If $i \in \{2,\dots,n-1\}$ and 
\eqref{eq:condition for sigma,2n,0i} is satisfied, 
then 
$\displaystyle 
I_{\eta,\Lambda,\sigma}
\simeq 
\begin{xy}
(-6,3.5)*{\overline{\pi}_{0,i-1}}="A_{11}",
(6,3.5)*{\overline{\pi}_{0,i+1}}="A_{12}",
(0,-3.5)*{\overline{\pi}_{0,i}}="A_{2}",
\ar "A_{11}";"A_{2}"
\ar "A_{12}";"A_{2}"
\end{xy}$. 
\item
If $i=n$ and 
\eqref{eq:condition for sigma,2n,0i} is satisfied, 
then 
$\displaystyle  
I_{\eta,\Lambda,\sigma}
\simeq 
\begin{xy}
(-9,4)*{\pi_{0}}="A_{11}",
(0,4)*{\overline{\pi}_{0,n-1}}="A_{12}",
(9,4)*{\pi_{1}}="A_{13}",
(0,-4)*{\overline{\pi}_{0,n}}="A_{2}",
\ar "A_{11}";"A_{2}"
\ar "A_{12}";"A_{2}"
\ar "A_{13}";"A_{2}"
\end{xy} 
$.
\end{enumerate}
\end{theorem}

\begin{theorem}
\label{theorem:main, 2n+1}
Suppose $G = Spin(2n+1,1)$ and the infinitesimal character 
$\Lambda=(\Lambda_{1}, \dots, \Lambda_{n+1})$ 
($\Lambda_{1} > \Lambda_{2} > \dots 
> \Lambda_{n} > |\Lambda_{n+1}|$) is regular integral. 
Let $\sigma$ be an irreducible representation of 
$M^{\eta} \simeq Spin(2n-1)$. 
\begin{enumerate}
\item
$\stWhLsmg$ is not zero if and only if the highest weight 
$\gamma = (\gamma_{1}, \dots, \gamma_{n-1})$ of 
$\sigma$ satisfies 
\begin{equation}\label{eq:condition for sigma,2n+1,0i}
\begin{cases}
\Lambda_{p}-n+p
\geq \gamma_{p} \geq 
\Lambda_{p+1}-n+p+1
\quad 
p=1,\dots,i-2, 
\\
\Lambda_{p+1}-n+p
\geq \gamma_{p} \geq 
\Lambda_{p+2}-n+p+1, 
\quad 
p=i-1,\dots,n-2, 
\\
\Lambda_{n}-1
\geq \gamma_{n-1} \geq 
|\Lambda_{n+1}| 
\quad 
\mbox{(if \enskip $i \leq n$)}
\end{cases}
\end{equation}
for some $i \in \{2,\dots,n+1\}$. 
\item
If $i \in \{2,\dots,n\}$ and 
\eqref{eq:condition for sigma,2n+1,0i} is satisfied, 
then 
$\displaystyle 
I_{\eta,\Lambda,\sigma}
\simeq 
\begin{xy}
(-6,3.5)*{\overline{\pi}_{0,i-1}}="A_{11}",
(6,3.5)*{\overline{\pi}_{0,i+1}}="A_{12}",
(0,-3.5)*{\overline{\pi}_{0,i}}="A_{2}",
\ar "A_{11}";"A_{2}"
\ar "A_{12}";"A_{2}"
\end{xy}$. 
\item
If $i = n+1$ and 
\eqref{eq:condition for sigma,2n+1,0i} is satisfied, 
then 
$\displaystyle 
I_{\eta,\Lambda,\sigma}
\simeq 
\begin{xy}
(0,3.5)*{\overline{\pi}_{0,n}}="A_{1}",
(0,-3.5)*{\overline{\pi}_{0,n+1}}="A_{2}",
\ar "A_{1}";"A_{2}"
\end{xy}$. 
\end{enumerate}
\end{theorem}

This paper is organized as follows. 
In \S\ref{section:Spin(r,1)}, 
we recall the structure of $Spin(r,1)$ and the
classification of irreducible Harish-Chandra modules of it. 
In \S\ref{section:composition factors}, 
we first show that $\stWhLsmg$ has a unique irreducible
submodule if it is non-zero. 
Also determined are 
the possible irreducible factors appearing in the composition series
of it. 
In order to determine the socle filtration of $\stWhLsmg$, 
we need to use the explicit formulas of $K$-type shift operators. 
Such operators are obtained in \S\ref{section:K-type shift operators}. 
In \S\ref{section:determination of composition series}, 
the socle filtration of $\stWhLsmg$ is completely
determined. 
The key tools for our calculation are 
Lemma~\ref{lemma:condition for vanishing under shift}, 
\ref{lemma:zero shifts} and Theorem~\ref{theorem:Ext}. 

Before going ahead, we introduce notation used in this paper. 
For a real Lie group $L$, the Lie algebra of it is denoted by
$\lier{l}$ and its complexification by 
$\lie{l} = \lier{l} \otimes_{\R} \C$. 
This notation will be applied to any Lie groups. 
For a compact Lie group $L$, the set of equivalence classes of
irreducible representations of $L$ is denoted by $\widehat{L}$. 
The representation space of $\pi \in \widehat{L}$ is denoted by
$V_{\pi}^{L}$. 
If $L$ is connected and $\pi$ is the irreducible representation
whose highest weight is $\lambda$, we also denote it by 
$V_{\lambda}^{L}$. 
For $\pi \in \widehat{L}$, the contragredient representation is
denoted by $\pi^{\ast}$, and if $\lambda$ is the highest weight of
$\pi$, then the highest weight of $\pi^{\ast}$ is denoted by
$\lambda^{\ast}$. 

Suppose that $K$ is a maximal compact subgroup of a real
reductive group $G$. 
For a $\brgK$-module $\pi$, the $K$-spectrum 
$\{\tau \in \widehat{K}\,|\, \tau \subset \pi|_{K}\}$ is denoted by 
$\widehat{K}(\pi)$. 

\section{The group $Spin(r,1)$ and its irreducible Harish-Chandra
  modules} 
\label{section:Spin(r,1)}

In the following of this paper, 
we put $G=Spin(r,1)$, $r \geq 3$, and the infinitesimal character
$\Lambda$ is assumed to be regular integral.

\subsection{Structure of $Spin(r,1)$} 
\label{subsection:structure of Spin(r,1)}

Denote by $E_{ij}$ the standard generators of
$\lie{gl}_{r+1}(\C)$ and define $A_{ij} = E_{ij}-E_{ji}$. 
The group $Spin(r,1)$ is the connected two-fold linear cover of
$SO_{0}(r,1)$. 
A maximal compact subgroup $K$ of $G$ is isomorphic to $Spin(r)$. 
Set 
\begin{align*}
\lier{k} 
&:= 
\left\{
\left.
\begin{pmatrix}
X & \vect{0} \\ {}^{t} \vect{0} & 0 
\end{pmatrix} 
\, 
\right| 
\, 
X \in \lie{so}(r)
\right\}, 
&
&
\lier{s} 
:= 
\left\{
\left.
\begin{pmatrix}
O & \I \vect{v} \\ -\I {}^{t} \vect{v} & 0 
\end{pmatrix} 
\, 
\right| 
\, 
\vect{v} \in \R^{r}
\right\}. 
\end{align*}
Then $\lier{g}=\lier{k}+\lier{s}$ realizes the Lie algebra of $G$, and
this is a Cartan decomposition of $\lier{g}$. 
Let 
\[
h := \I A_{r+1,r}
\qquad 
\lier{a} 
:= 
\R h, 
\]
and define $f \in \lier{a}^{\ast}$ by $f(h) = 1$. 
Then $\lier{a}$ is a maximal abelian subspace of $\lier{s}$. 
The restricted root system 
$\roots{\lier{g}}{\lier{a}}$ is $\{\pm f\}$. 
Choose a positive system 
$\proots{\lier{g}}{\lier{a}} 
= 
\{f\}$, 
and denote the corresponding nilpotent subalgebra 
$(\lier{g})_{f}$ by $\lier{n}$. 
One obtains an Iwasawa decomposition 
\begin{align*}
& 
\lier{g} = \lier{k} + \lier{a} + \lier{n}, 
& 
& 
G = K A N, 
\end{align*}
where $A = \exp \lier{a}$ and $N = \exp \lier{n}$. 
Let 
\begin{equation}\label{eq:basis of u}
X_{i} 
:= 
A_{r,i} + \I A_{r+1,i}
\qquad 
(1 \leq i \leq r-1). 
\end{equation}
Then 
$\{X_{i} \,|\, 1 \leq i \leq r-1\}$ is a basis of
$\lier{n}$. 

In our $Spin(r,1)$ case, 
$M$ is isomorphic to $Spin(r-1)$. 
It acts on the space of non-degenerate unitary characters of $N$ by 
$\eta \mapsto \eta^{m}(n) := \eta(m^{-1}nm)$, $m \in M$. 
Therefore, we may choose a manageable unitary character when we
calculate Whittaker modules. 
We use the non-degenerate character $\eta$ defined by 
\begin{align}\label{eq:definition of psi}
&
\eta(X_{i}) = 0, \quad i=1,\dots,r-2,
&
& 
\eta(X_{r-1}) = \I \xi, \quad \xi > 0. 
\end{align}
It is easy to see that $M^{\eta}$ is isomorphic to 
$Spin(r-2)$.

\subsection{Classification of irreducible Harish-Chandra modules}
\label{subsection:classification of (g,K)-modules}

We review the classification of irreducible Harish-Chandra modules of
$G=Spin(r,1)$ with regular integral infinitesimal character. 
For details, see \cite{C} for example. 
We use the notation $\pi_{0,i}$, $\overline{\pi}_{0,i}$ etc in
\cite{C}. 

For an irreducible representation $\delta$ of $M$ and 
an element $\nu \in \lie{a}^{\ast}$, 
let $X_{P}(\delta,\nu)$ be the Harish-Chandra module of 
the principal series representation 
$\Ind_{P}^{G}(\delta \otimes e^{\nu+\rho_{A}})$. 
Here, $P=MAN$ is a minimal parabolic subgroup of $G$ and 
$\rho_{A} := \frac{1}{2} \tr (\ad_{\lie{a}}|_{\lie{n}}) 
\in \lie{a}^{\ast}$.

Firstly, consider the case $r=2n$, $n \geq 2$. 
There are two conjugacy classes of Cartan subgroups in $G$,
one is compact and the other is maximally split.  
Let $\lie{h}_{c}$ be the complexified Cartan subalgebra spanned by 
$\I A_{2i,2i-1}$, $i=1,\dots,n$, 
and let $H_{c}$ be the corresponding compact Cartan subgroup.  
Define a basis  $\{\epsilon_{i}\,|\, i=1,\dots,n\}$ of $\lie{h}_{c}^{\ast}$ 
by $\epsilon_{i}(\I A_{2j,2j-1}) = \delta_{ij}$ 
(Kronecker's delta). 
Choose a maximally split Cartan subgroup $H_{s}:=(H_{c}\cap M) A$. 
The complexified Lie algebra $\lie{h}_{s}$ of it is the linear span of 
$\I A_{2i, 2i-1}$ ($i=1,\dots,n-1$) and $h = \I A_{2n+1,2n}$, 
so $\epsilon_{i}$, $i=1,\dots,n-1$ and $f$ form a basis of
$\lie{h}_{s}^{\ast}$.

Consider the irreducible Harish-Chandra modules with the
regular integral infinitesimal character $\Lambda$, 
which is conjugate to 
\begin{equation}\label{eq:Lambda,2n}
\sum_{p=1}^{n} \Lambda_{p} \epsilon_{p} \in \lie{h}_{c}^{\ast}, 
\quad 
\Lambda_{p} \in \frac{1}{2} \Z, 
\quad 
\Lambda_{p} - \Lambda_{p+1} \in \Z_{>0} 
\mbox{ for } p=1, \dots, n-1, 
\mbox{ and } \Lambda_{n} > 0.  
\end{equation}
There are two inequivalent discrete series representations
$\pi_{i}$, $i=0, 1$, whose Harish-Chandra parameters are 
\begin{align*}
&
\sum_{p=1}^{n} \Lambda_{p} \epsilon_{p}, 
&
&
\sum_{p=1}^{n-1} \Lambda_{p} \epsilon_{p} - \Lambda_{n} \epsilon_{n}, 
\end{align*}
respectively. 

Since $W_{\lie{g}} \simeq S_{n} \ltimes \Z_{2}^{n}$ and 
$W(G,H_{s}) \simeq S_{n-1} \ltimes \Z_{2}^{n}$, 
and since $H_{s}$ is connected, 
there are $n$ equivalence classes of non-tempered irreducible
representations of $Spin(2n,1)$. 
For $i=1,\dots,n$, 
define $\mu_{0,i} \in (\lie{h}_{s} \cap \lie{m})^{\ast}$ and 
$\nu_{0,i} \in \lie{a}^{\ast}$ by 
\begin{align}\label{eq:M-pseudocharacter,2n}
\mu_{0,i} := 
& \sum_{p=1}^{i-1} \Lambda_{p} \epsilon_{p} 
+ \sum_{p=i}^{n-1} \Lambda_{p+1} \epsilon_{p} 
- \rhom, 
&
\nu_{0,i} :=& 
\Lambda_{i} f, 
\end{align}
where 
$\rhom :=\frac{1}{2} \sum_{p=1}^{n-1} (2n-1-2p) \epsilon_{p}$. 
Let $\delta_{0,i}$ be the irreducible representation of $M$ with
the highest weight $\mu_{0,i}$, 
and let 
$\pi_{0,i} := X_{P}(\delta_{0,i}, \nu_{0,i})$. 
Then $\pi_{0,i}$ has the unique irreducible quotient, which we denote
by $\overline{\pi}_{0,i}$. 

For an irreducible $\brgK$-module $\pi$, 
let $\ell(\pi)$ be the length of $\pi$ defined in 
\cite[Definition~8.1.4]{V2}. 
The classification of irreducible $\brgK$-modules of $Spin(2n,1)$ is
as follows (see \cite{C}, for example). 

\begin{theorem}
\label{theorem:classification of irr modules, 2n}
The irreducible Harish-Chandra modules of $Spin(2n,1)$ with the regular
integral infinitesimal character $\Lambda$ are parametrized 
by the set 
\[
\{\pi_{0}, \pi_{1}\} \cup \{\overline{\pi}_{0,i}\,|\, i=1,\dots,n\}. 
\]
The lengths of $\pi_{0}$, $\pi_{1}$ and $\overline{\pi}_{0,i}$, 
$i=1,\dots,n$, are $0$, $0$ and $n-i+1$, respectively. 
\end{theorem}

In order to state the composition series, we use diagrammatic
expression. 
\begin{definition}\label{definition:diagram of composition series} 
Suppose $A_{1}, A_{2}$ are distinct composition factors of a 
$\brgK$-module $V$. 
If there exist elements $\{v_{i}\} \subset A_{1}$ and 
$\{X_{i}\} \subset \lie{g}$ such that $\sum_{i} X_{i} v_{i}$ is
non-zero and contained in $A_{2}$, 
then we connect $A_{1}$ and $A_{2}$ by an arrow 
$A_{1} \rightarrow A_{2}$. 
\end{definition}

\begin{theorem}[\cite{C}]\label{composition series of PS,2n} 
The socle filtrations of $\pi_{0,i}$ are 
\begin{align*}
&
\pi_{0,i} 
\quad \simeq \quad 
\begin{xy}
(0,3.5)*{\overline{\pi}_{0,i}}="A_{1}",
(0,-3.5)*{\overline{\pi}_{0,i+1}}="A_{2}",
\ar "A_{1}";"A_{2}"
\end{xy} 
\quad 
\mbox{if}
\quad 
i=1,\dots,n-1, 
\quad 
\mbox{and}
&
&
\pi_{0,n} 
\quad \simeq \quad 
\begin{xy}
(0,3.5)*{\overline{\pi}_{0,i}}="A_{1}",
(-6,-3.5)*{\pi_{0}}="A_{21}",
(6,-3.5)*{\pi_{1}}="A_{22}",
\ar "A_{1}";"A_{21}"
\ar "A_{1}";"A_{22}"
\end{xy}. 
\end{align*}
\end{theorem}

The Blattner formula gives the $K$-spectra of the discrete series
representations $\pi_{0}, \pi_{1}$. 
Starting from the discrete series, we obtain the
$K$-spectrum of $\overline{\pi}_{0,i}$ inductively, by using
Theorem~\ref{composition series of PS,2n}. 
To state the theorem, let $\Lambda_{0} := \infty$. 
%
\begin{theorem}
\label{theorem:K-spectra,2n} 
\begin{enumerate}
\item
The $K$-spectra of $\pi_{0}$ and $\pi_{1}$ are 
\begin{align}
\widehat{K}(\pi_{0}) 
= 
\{(\tau_{\lambda}, V_{\lambda}^{K}) \,|\, 
\,&  
\Lambda_{p-1} -n+p-\frac{1}{2}
\geq \lambda_{p} \geq
\Lambda_{p}-n+p+\frac{1}{2},
\enskip
(1 \leq p \leq n)\}, 
\label{eq:K-spectrum of pi_0}\\
\widehat{K}(\pi_{1}) 
= 
\{(\tau_{\lambda}, V_{\lambda}^{K}) \,|\, 
\,&  
\Lambda_{p-1} -n+p-\frac{1}{2}
\geq \lambda_{p} \geq
\Lambda_{p}-n+p+\frac{1}{2},
\enskip
(1 \leq p \leq n-1); 
\notag\\
& 
-\Lambda_{n} -\frac{1}{2}
\geq \lambda_{n} \geq
-\Lambda_{n-1}+\frac{1}{2}
\}. 
\label{eq:K-spectrum of pi_1}
\end{align}
\item
For $i=1,\dots,n$, the $K$-spectrum of $\overline{\pi}_{0,i}$ is 
\begin{align}
\widehat{K}(\overline{\pi}_{0,i}) 
= 
\{(\tau_{\lambda}, V_{\lambda}^{K}) \,|\, 
\, & 
\Lambda_{p-1} -n+p-\frac{1}{2} 
\geq \lambda_{p} \geq
\Lambda_{p}-n+p+\frac{1}{2}, 
\enskip
(1 \leq p \leq i-1); 
\notag\\
& 
\Lambda_{p} -n+p-\frac{1}{2}
\geq \lambda_{p} \geq
\Lambda_{p+1}-n+p+\frac{1}{2},
\enskip 
(i \leq p \leq n-1); 
\label{eq:K-spectrum of pi_0i,2n}\\
& 
\Lambda_{n} -\frac{1}{2} 
\geq |\lambda_{n}| 
\}. 
\notag\end{align}
\end{enumerate}
In each case, every $K$-type occurs in $\overline{\pi}_{0,i}$ with
multiplicity one. 
\end{theorem}

Secondly, consider the case $r=2n+1$, $n \geq 1$. 
There are only one conjugacy class of Cartan subgroups in $G$. 
Let $\lie{h}_{s}$ be the complexified Cartan subalgebra spanned by 
$\I A_{2i,2i-1}$, $i=1,\dots,n$ and $h=\I A_{2n+2,2n+1}$, 
and let $H_{s}$ be the corresponding Cartan subgroup. 
Define $\{\epsilon_{i}\,|\, i=1,\dots,n\}$ as in the $r=2n$ case.

Consider the irreducible Harish-Chandra modules with the
regular integral infinitesimal character $\Lambda$, 
which is conjugate to 
\begin{align}
\sum_{p=1}^{n} &\Lambda_{p} \epsilon_{p} 
+ \Lambda_{n+1} f \in \lie{h}_{s}^{\ast}, 
\quad 
\Lambda_{p} \in \frac{1}{2} \Z, 
\label{eq:Lambda,2n+1}\\
&\mbox{with} \quad 
\Lambda_{p} - \Lambda_{p+1} \in \Z_{>0} 
\mbox{ for } p=1, \dots, n, 
\mbox{ and } \Lambda_{n}+\Lambda_{n+1} \in \Z_{>0}.  
\notag
\end{align}
Since $W_{\lie{g}} \simeq S_{n+1} \ltimes \Z_{2}^{n}$ and 
$W(G,H_{s}) \simeq S_{n} \ltimes \Z_{2}^{n}$, 
and since $H_{s}$ is connected, 
there are $(n+1)$ equivalence classes of irreducible representations
of $Spin(2n+1,1)$. 
For $i=1,\dots,n+1$, 
define $\mu_{0,i} \in (\lie{h}_{s} \cap \lie{m})^{\ast}$ and 
$\nu_{0,i} \in \lie{a}^{\ast}$ by 
\begin{align}\label{eq:M-pseudocharacter}
&
\mu_{0,i} := 
\sum_{p=1}^{i-1} \Lambda_{p} \epsilon_{p} 
+ \sum_{p=i}^{n} \Lambda_{p+1} \epsilon_{p} 
- \rhom, 
&
\nu_{0,i} :=& 
\Lambda_{i} f, 
\end{align}
where 
$\rhom :=\sum_{p=1}^{n-1} (n-p) \epsilon_{p}$. 
Let $\delta_{0,i}$ be the irreducible representation of $M$ with
the highest weight $\mu_{0,i}$, 
and let 
$\pi_{0,i} := X_{P}(\delta_{0,i}, \nu_{0,i})$. 
Then $\pi_{0,i}$ has the unique irreducible quotient, which we denote
by $\overline{\pi}_{0,i}$. 
\begin{theorem}
\label{composition series of PS,2n+1} 
The irreducible Harish-Chandra modules of $Spin(2n+1,1)$ with the
regular integral infinitesimal character $\Lambda$ are parametrized 
by the set 
\[
\{\overline{\pi}_{0,i}\,|\, i=1,\dots,n+1\}. 
\]
The lengths of $\overline{\pi}_{0,i}$, $i=1,2,\dots,n+1$, are $n-i+1$,
respectively. 

The socle filtrations of $\pi_{0,i}$ are 
\begin{align*}
&
\pi_{0,i} 
\quad \simeq \quad 
\begin{xy}
(0,3.5)*{\overline{\pi}_{0,i}}="A_{1}",
(0,-3.5)*{\overline{\pi}_{0,i+1}}="A_{2}",
\ar "A_{1}";"A_{2}"
\end{xy} 
\quad 
\mbox{if}
\quad 
i=1,\dots,n, 
\quad 
\mbox{and}
&
&
\pi_{0,n+1} 
= 
\overline{\pi}_{0,n+1}
\quad
\mbox{is irreducible.} 
\end{align*}
\end{theorem}

Starting from $\pi_{0,n+1} = \overline{\pi}_{0,n+1}$, 
we obtain the
$K$-spectrum of $\overline{\pi}_{0,i}$ inductively, by using
Theorem~\ref{composition series of PS,2n+1}. 
As before, let $\Lambda_{0} := \infty$. 
\begin{theorem}
\label{theorem:K-spectra,2n+1} 
For $i=1,\dots,n+1$, the $K$-spectrum of $\overline{\pi}_{0,i}$ is 
\begin{align}
\widehat{K}(\overline{\pi}_{0,i}) 
= 
\{(\tau_{\lambda}, V_{\lambda}^{K}) \,|\, 
\, & 
\Lambda_{p-1} -n+p-1
\geq \lambda_{p} \geq
\Lambda_{p}-n+p,
\enskip
(1 \leq p \leq i-1); 
\notag\\
& 
\Lambda_{p} -n+p-1
\geq \lambda_{p} \geq
\Lambda_{p+1}-n+p,
\enskip 
(i \leq p \leq n-1);
\notag\\
& 
\Lambda_{n} -1
\geq \lambda_{n} \geq
|\Lambda_{n+1}|, 
\enskip 
(\mbox{if } \enskip i<n+1)\}. 
\label{eq:K-spectrum of pi_0i,2n+1}
\end{align}
In each case, every $K$-type occurs in $\overline{\pi}_{0,i}$ with
multiplicity one. 
\end{theorem}


\section{Composition factors of $\stWhLsmg$}
\label{section:composition factors}

In this section we first determine the submodules of
$\stWhLsmg$. 
It is known by \cite{M1} that a $\brgK$-module $V$ can be a submodule
of $C^{\infty}(G/N;\eta)_{K}$ if and only if the Gelfand-Kirillov
dimension $\Dim V$ of it is equal to $\dim N$. 
For $G=Spin(r,1)$, 
$\Dim \overline{\pi}_{0,1} = 0$ and the Gelfand-Kirillov dimensions of
other irreducible modules are all $\dim N$. See \cite{C} for example. 
Therefore, an irreducible submodule of $\stWhLsmg$ is isomorphic to
one of $\pi_{0}, \pi_{1}$ or $\overline{\pi}_{0,i}$, 
$i=2,\dots,n+1$.

\subsection{Unique simple submodule} 
\label{subsection:Unique simple submodule}
By the discussion in \cite[\S4.2]{T2}, 
the following lemma holds. 
\begin{lemma}\label{lemma:condition for sub}
Let $(\pi, V)$ be an irreducible Harish-Chandra module with 
$\Dim V = \dim N$. 
Let $\{X_{P}(\delta_{p},\nu_{p})\, | \, p=1,\dots,k\}$ 
be the set of principal series representations which contain 
$(\pi,V)$ as a subquotient. 
If $(\pi,V)$ is a submodule of $\stWhLsmg$, 
then $\sigma$ is a submodule of $\delta_{p}|_{M^{\eta}}$ for every 
$p=1,\dots,k$. 

Conversely, for $\sigma \in \widehat{M^{\eta}}$, 
suppose that there exists a principal series 
$X_{P}(\delta,\nu)$ with infinitesimal character $\Lambda$ which
satisfies $\sigma \subset \delta|_{M^{\eta}}$. 
Then, $\stWhLsmg$ is non-zero. 
\end{lemma}


By this lemma, we can determine the non-zero standard Whittaker
$\brgK$-modules and their subrepresentations. 
\begin{proposition}
\label{proposition:condition for sigma, integral,2n} 
Suppose $r=2n$ and the infinitesimal character $\Lambda$ is regular 
integral. 
Let $\gamma = (\gamma_{1},\dots, \gamma_{n-1})$ be the highest weight
of the irreducible representation $\sigma$ of 
$M^{\eta} \simeq Spin(2n-2)$. 
\begin{enumerate}
\item
The irreducible module  
$\pi_{0}$ (resp. $\pi_{1}$) is a submodule of $\stWhLsmg$ 
if and only if $\gamma$ 
satisfies \eqref{eq:condition for sigma,2n,0} 
(resp. \eqref{eq:condition for sigma,2n,1}). 
%
%
\item
The irreducible module  
$\overline{\pi}_{0,i}$, $i=2,\dots,n$, is a submodule of $\stWhLsmg$ 
if and only if $\gamma$ 
satisfies \eqref{eq:condition for sigma,2n,0i}. 
\end{enumerate}
Especially, $\stWhLsmg$ is non-zero 
if and only if the highest weight of
$\sigma$ satisfies one of the conditions 
\eqref{eq:condition for sigma,2n,0}, 
\eqref{eq:condition for sigma,2n,1} 
or \eqref{eq:condition for sigma,2n,0i} for some $i=2,\dots,n$. 
In these cases, $\pi_{0}$, $\pi_{1}$ or $\overline{\pi}_{0,i}$ is the
unique simple submodule of $\stWhLsmg$. 
\end{proposition}
\begin{proof}
We first show (2). 
By Theorem~\ref{composition series of PS,2n}, 
$\overline{\pi}_{0,i}$, $i=2,\dots,n$, is a composition factor of the
principal series $\pi_{0,k}$ if and only if 
$k = i$ or $i-1$. 
Therefore, if $\overline{\pi}_{0,i}$ is a submodule of
$\stWhLsmg$, then 
$\sigma 
\subset 
\delta_{0,i}|_{M^{\eta}}$ and 
$\sigma \subset \delta_{0,i-1}|_{M^{\eta}}$. 
Conversely, if $\sigma$ satisfies this condition, 
then $\stWhLsmg$ is non-zero, by Lemma~\ref{lemma:condition for sub}. 

Recall the branching rule for the restriction of an irreducible
representation of $Spin(2n-1)$ to $Spin(2n-2)$. 
For an irreducible representation $\delta_{\mu}$ of $Spin(2n-1)$ with the
highest weight $\mu=(\mu_{1},\dots,\mu_{n-1})$, 
the restriction $\delta_{\mu}|_{Spin(2n-2)}$ is a direct sum of
$\sigma' \in Spin(2n-2)\widehat{\ }$, whose highest
weight $\gamma = (\gamma_{1},\dots,\gamma_{n-1})$ satisfies 
\begin{align*}
& \mu_{p} \geq \gamma_{p} \geq \mu_{p+1}, 
\quad p=1,\dots,n-2, 
& 
& \mu_{n-1} \geq |\gamma_{n-1}|, 
& 
& \mu_{p} - \gamma_{p} \in \Z. 
\end{align*}
It follows that the restriction 
$\delta_{0,k} \in \widehat{M}$ 
to $M^{\eta}$ is a direct
sum of $\sigma' \in \widehat{M^{\eta}}$, 
whose highest weight 
$\gamma = (\gamma_{1},\dots,\gamma_{n-1})$ satisfies 
\[
\begin{cases}
\Lambda_{p}-n+p+\frac{1}{2} 
\geq \gamma_{p} \geq 
\Lambda_{p+1}-n+p+\frac{3}{2}, 
\quad 
p=1,\dots,k-2, 
\\
\Lambda_{k-1}-n+k-\frac{1}{2} 
\geq \gamma_{k-1} \geq 
\Lambda_{k+1}-n+k+\frac{1}{2}, 
\\
\Lambda_{p+1}-n+p+\frac{1}{2} 
\geq \gamma_{p} \geq 
\Lambda_{p+2}-n+p+\frac{3}{2}, 
\quad 
p=k,\dots,n-2, 
\\
\Lambda_{n}-\frac{1}{2} 
\geq |\gamma_{n-1}|. 
\end{cases}
\]
(
If
$k=n$, then the second and the third lines are omitted.) 
Therefore, when $2 \leq i \leq n$, 
$\sigma \in \widehat{M^{\eta}}$ satisfies 
$\sigma 
\subset 
\delta_{0,i}|_{M^{\eta}}$ and 
$\sigma \subset \delta_{0,i-1}|_{M^{\eta}}$ 
if and only if the highest weight $\gamma$ of $\sigma$ satisfies 
\eqref{eq:condition for sigma,2n,0i}. 
This proves the ``only if'' part of proposition. 

The proofs of the ``if'' part and the uniqueness of the socle of
$\stWhLsmg$ are the same as those of \cite[Proposition~4.2]{T2}, 
so we omit them here.

The proof of (1) is almost the same as that of (2). 
We can show that $\pi_{0}$ or $\pi_{1}$ is a submodule of $\stWhLsmg$
only if $\gamma$ satisfies \eqref{eq:condition for sigma,2n,0} or 
\eqref{eq:condition for sigma,2n,1}. 
But the method used here does not tell us the signature condition for
$\gamma_{n-1}$ in \eqref{eq:condition for sigma,2n,0} and 
\eqref{eq:condition for sigma,2n,1}. 
To complete the proof, 
we need to write explicitly the Whittaker functions
characterizing the submodule of $\stWhLsmg$. 
This will be done 
in \S\ref{section:K-type shift operators} 
(Lemma~\ref{lemma:pi_0,pi_1 sub}). 
\end{proof}

Just in the same way, we can determine the submodules of
$\stWhLsmg$ in the case $r=2n+1$. 
\begin{proposition}
\label{proposition:condition for sigma, integral,2n+1} 
Suppose $r=2n+1$ and the regular infinitesimal character $\Lambda$ is
integral. 
Let $\gamma = (\gamma_{1},\dots, \gamma_{n-1})$ be the highest weight
of the irreducible representation $\sigma$ of 
$M^{\eta} \simeq Spin(2n-1)$. 
Then the irreducible module  
$\overline{\pi}_{0,i}$, $i=2,\dots,n+1$, is a submodule of $\stWhLsmg$ 
if and only if $\gamma$ 
satisfies \eqref{eq:condition for sigma,2n+1,0i}. 
Especially, $\stWhLsmg$ is non-zero 
if and only if the highest weight of
$\sigma$ satisfies the condition 
\eqref{eq:condition for sigma,2n+1,0i} for some $i=2,\dots,n+1$. 
In these cases, $\overline{\pi}_{0,i}$ is the unique simple submodule of
$\stWhLsmg$. 
\end{proposition}

\subsection{Composition factors}\label{subsection:composition factors} 

Hereafter, we denote $\stWhLsmg$ by 
$\stWhLgmg$ if the highest weight of $\sigma$ is
  $\gamma$. 
We also denote by $\sigma_{\gamma}$ the irreducible representation of
$M^{\eta}$ whose highest weight is $\gamma$. 
We determine the irreducible representations appearing in the
composition series of $\stWhLgmg$. 
\begin{proposition}\label{proposition:composition factors-1} 
\begin{enumerate}
\item
Suppose $r=2n$ and $\gamma$ satisfies 
\eqref{eq:condition for sigma,2n,0} 
or \eqref{eq:condition for sigma,2n,1}. 
Then, an irreducible composition factor of $\stWhLgmg$ is
isomorphic to $\pi_{0}$, $\pi_{1}$ or $\overline{\pi}_{0,n}$. 
\item
Suppose $r=2n$, $i\in\{2,\dots,n-1\}$ and $\gamma$ satisfies 
\eqref{eq:condition for sigma,2n,0i}. 
Then, an irreducible composition factor of $\stWhLgmg$ is
isomorphic to $\overline{\pi}_{0,i-1}$, $\overline{\pi}_{0,i}$ 
or $\overline{\pi}_{0,i+1}$. 
\item
Suppose $r=2n$ and $\gamma$ satisfies 
\eqref{eq:condition for sigma,2n,0i} for $i=n$. 
Then, an irreducible composition factor of $\stWhLgmg$ is
isomorphic to $\overline{\pi}_{0,n-1}$, $\overline{\pi}_{0,n}$, 
$\pi_{0}$ or $\pi_{1}$. 
\item
Suppose $r=2n+1$, $i\in\{2,\dots,n\}$ and $\gamma$ satisfies 
\eqref{eq:condition for sigma,2n+1,0i}. 
Then, an irreducible composition factor of $\stWhLgmg$ is
isomorphic to $\overline{\pi}_{0,i-1}$, $\overline{\pi}_{0,i}$ 
or $\overline{\pi}_{0,i+1}$. 
\item
Suppose $r=2n+1$ and $\gamma$ satisfies 
\eqref{eq:condition for sigma,2n+1,0i} for $i=n+1$. 
Then, an irreducible composition factor of $\stWhLgmg$ is
isomorphic to $\overline{\pi}_{0,n}$ or $\overline{\pi}_{0,n+1}$. 
\end{enumerate}
\end{proposition}
\begin{proof}
We will show (2). 
The proofs of (1), (3), (4) and (5) are the same. 

Assume that $\gamma$ satisfies \eqref{eq:condition for sigma,2n,0i}. 
Since $\stWhLgmg$ is induced from the representation 
$\sigma_{\gamma} \otimes \eta$ of $M^{\eta} N$, 
each $K$-type of a composition factor of $\stWhLgmg$ must contain the
representation $\sigma_{\gamma}$. 
By \eqref{eq:K-spectrum of pi_0i,2n} and \eqref{eq:condition for sigma,2n,0i}, 
this is possible if and only if the composition factor is isomorphic
to $\overline{\pi}_{0,i-1}$, $\overline{\pi}_{0,i}$ or 
$\overline{\pi}_{0,i+1}$. 
\end{proof}

\begin{proposition}\label{proposition:at least one}
Suppose that one of the following conditions is satisfied: 
\begin{enumerate}
\item
$r=2n$, $\gamma$ satisfies 
\eqref{eq:condition for sigma,2n,0} 
or \eqref{eq:condition for sigma,2n,1}, 
and $\pi$ is $\pi_{0}$, $\pi_{1}$ or $\overline{\pi}_{0,n}$.   
\item
$r=2n$, $i\in\{2,\dots,n-1\}$, $\gamma$ satisfies 
\eqref{eq:condition for sigma,2n,0i}, 
and $\pi$ is $\overline{\pi}_{0,i-1}$, $\overline{\pi}_{0,i}$ 
or $\overline{\pi}_{0,i+1}$. 
\item
$r=2n$, $\gamma$ satisfies 
\eqref{eq:condition for sigma,2n,0i} for $i=n$, 
and $\pi$ is $\overline{\pi}_{0,n-1}$, $\overline{\pi}_{0,n}$,
$\pi_{0}$ or $\pi_{1}$. 
\item
$r=2n+1$, $i\in\{2, \dots, n\}$, $\gamma$ satisfies 
\eqref{eq:condition for sigma,2n+1,0i}. 
and $\pi$ is $\overline{\pi}_{0,i-1}$, $\overline{\pi}_{0,i}$ 
or $\overline{\pi}_{0,i+1}$. 
\item
$r=2n+1$,  $\gamma$ satisfies 
\eqref{eq:condition for sigma,2n+1,0i} for $i=n+1$, 
and $\pi$ is $\overline{\pi}_{0,n}$ or $\overline{\pi}_{0,n+1}$. 
\end{enumerate}
Then the multiplicity of $\pi$ in $\stWhLgmg$ is at least one. 
\end{proposition}
\begin{proof}
This can be shown just in the same way as \cite[Proposition~4.4]{T2}. 
\end{proof}


\section{$K$-type shift operators}
\label{section:K-type shift operators}

In order to determine the socle filtration of
$\stWhLgmg$, we need to write the actions of elements in
$\lie{s}$ on this space explicitly. 
This is achieved by the $K$-type shift operators. 
The contents of this section are almost the same as those of
\cite[\S5]{T2}. 
So we do not repeat the explanation and refer the readers to this
paper.

\subsection{Gelfand-Tsetlin basis} 
\label{subsection:Gelfand-Tsetlin basis} 
In order to write the $K$-type shift operators explicitly, 
we realize the $K$-types 
by using the Gelfand-Tsetlin basis (\cite{GT}). 

\begin{definition}\label{definition:GT}
Let $\lambda = (\lambda_{1}, \dots, \lambda_{\lfloor r/2 \rfloor})$ be
a dominant integral weight of $Spin(r)$. 
A {\it $\lambda$-Gelfand-Tsetlin pattern} is a set of vectors 
$Q=(\vect{q}_{1}, \dots, \vect{q}_{r-1})$ such that 
\begin{enumerate}
\item
$\vect{q}_i=(q_{i,1}, q_{i,2}, \dots, q_{i,\lfloor(i+1)/2\rfloor})$. 
\item
The numbers $q_{i,j}$ are all integers or all half integers. 
\item
$q_{2i+1,j} \geq q_{2i,j} \geq q_{2i+1,j+1}$, 
for any $j=1, \dots, i-1$.  
\item
$q_{2i+1,i} \geq q_{2i,i} \geq |q_{2i+1,i+1}|$.  
\item
$q_{2i,j} \geq q_{2i-1,j} \geq q_{2i,j+1}$, 
for any $j=1, \dots, i-1$. 
\item
$q_{2i,i} \geq q_{2i-1,i} \geq -q_{2i,i}$.  
\item
$q_{r-1,j}=\lambda_{j}$. 
\end{enumerate}
Here, $\lfloor a \rfloor$ is the largest integer not greater than
$a$. 
The set of all $\lambda$-Gelfand-Tsetlin patterns is denoted by
$GT(\lambda)$. 
\end{definition}
\begin{notation}
For any set or number $\ast$ depending on $Q \in GT(\lambda)$, 
we denote it by $\ast(Q)$, if we need to specify $Q$. 
For example, $q_{i,j}(Q)$ is the $q_{i,j}$ part of $Q \in GT(\lambda)$. 
\end{notation}
\begin{theorem}[\cite{GT}] 
For a dominant integral weight $\lambda$ of $Spin(r)$, 
let $(\tau_{\lambda}, V_{\lambda}^{Spin(r)})$ be the irreducible
representation of $Spin(r)$ with the highest weight $\lambda$. 
Then $GT(\lambda)$ is identified with a basis of 
$(\tau_{\lambda}, V_{\lambda}^{Spin(r)})$. 
\end{theorem}
The action of the element $A_{p,q} \in \lie{so}(r)$ is expressed as follows. 
For $j > 0$, let 
\begin{align*}
& 
l_{2i-1,j} := q_{2i-1,j} + i - j, 
& 
& 
l_{2i-1,-j} := - l_{2i-1,j}, 
\\
& 
l_{2i,j} := q_{2i,j} + i + 1 - j, 
& 
& 
l_{2i,-j} := - l_{2i,j} + 1,  
\end{align*}
and let $l_{2i,0} = 0$. 
Define $a_{p,q}(Q)$ by 
\begin{align*} 
a_{2i-1,j}(Q) 
&= 
\sgn j \, 
\sqrt{-
\frac{\prod_{1 \leq |k| \leq i-1}(l_{2i-1,j} + l_{2i-2, k}) 
\prod_{1 \leq |k| \leq i} (l_{2i-1,j} + l_{2i, k})}
{4 \prod_{\genfrac{}{}{0pt}{}{1 \leq |k| \leq i,}{k \not= \pm j}} 
(l_{2i-1,j} + l_{2i-1,k}) (l_{2i-1,j} + l_{2i-1,k} + 1)}
},
\intertext{for $j = \pm 1, \dots, \pm i$, and}  
a_{2i,j}(Q) 
&= 
\varepsilon_{2i,j}(Q) 
\sqrt{-
\frac{\prod_{1 \leq |k| \leq i}(l_{2i,j} + l_{2i-1, k}) 
\prod_{1 \leq |k| \leq i+1} (l_{2i,j} + l_{2i+1, k})}
{(4 l_{2i,j}^{2} - 1) 
\prod_{\genfrac{}{}{0pt}{}{0 \leq |k| \leq i}{k \not= \pm j}} 
(l_{2i,j} + l_{2i,k}) (l_{2i,j} - l_{2i,k})}
},
\end{align*}
for $j = 0, \pm 1, \dots, \pm i$, 
where $\varepsilon_{2i,j}(Q)$ is $\sgn j$ if $j \not= 0$, 
and $\sgn(q_{2i-1,i}\, q_{2i+1,i+1})$ if $j = 0$. 

Let $\sigma_{a,b}$ be the shift operator, sending $\vect{q}_a$ to 
$\vect{q}_{a} + (0, \dots, \overset{|b|}{\sgn(b)}, 0 ,\dots, 0)$. 

\begin{theorem}[\cite{GT}]
Under the above notation, the action of the Lie algebra is expressed
as 
\begin{align*}
\tau_\lambda(A_{2i+1,2i})Q
&=
\sum_{1 \leq |j| \leq i} a_{2i-1,j}(Q)\, \sigma_{2i-1,j}Q, 
\\
\tau_\lambda(A_{2i+2,2i+1})Q
&=
\sum_{0 \leq |j| \leq i} a_{2i,j}(Q)\, \sigma_{2i,j}Q. 
\end{align*}
\end{theorem}

\begin{remark}\label{remark:GT restriction}
The Gelfand-Tsetlin basis is compatible with the restriction to
smaller groups $Spin(k)$, $k=1,\dots,r-1$. 
More precisely, the restriction of $\tau_{\lambda}$ to $Spin(r-1)$ is
multiplicity free, and the highest weights of the irreducible
representation appearing in $\tau_{\lambda}|_{Spin(r-1)}$ are the above 
$\vect{q}_{r-2}$'s. 
Moreover, 
the vector $Q = (\vect{q}_{1}, \dots, \vect{q}_{r-2}, \vect{q}_{r-1})$
is contained in the $\vect{q}_{r-2}$-isotypic subspace of 
$\tau_{\lambda}|_{Spin(r-1)}$. 
\end{remark}

\begin{remark}\label{remark:GT of dual}
The highest weight $\lambda^{\ast}$ of the contragredient representation 
$(\tau_{\lambda^{\ast}}, V_{\lambda^{\ast}}^{K})$ of 
$(\tau_{\lambda}, V_{\lambda}^{K})$ is 
$\lambda^{\ast} = (\lambda_{1}, \dots, (-1)^{n} \lambda_{n})$ if
$r=2n$, and 
$\lambda^{\ast} = \lambda$ if $r=2n+1$. 
In this case, 
$Q^{\ast} := (\vect{q}_{1}^{\ast}, \dots, \vect{q}_{r-1}^{\ast}) 
\in GT(\lambda^{\ast})$, 
$\vect{q}_{2i+1}^{\ast} := \vect{q}_{2i+1}$, 
$\vect{q}_{2i}^{\ast} 
:= (q_{2i,1}, \dots, q_{2i,i-1},(-1)^{i} q_{2i,i})$ is dual to
$Q \in GT(\lambda)$. 
\end{remark}

\subsection{Shift operators}\label{subsection:shift operators} 
Choose a $K$-type $(\tau_{\lambda},V_{\lambda}^{K})$ of
$\stWhLgmg$, whose highest weight is $\lambda$. 
By appropriately defining the action of $\lie{g}$ on 
$C^{\infty}(A \rightarrow 
\Hom_{M^{\eta}}(V_{\lambda}^{K}, V_{\gamma}^{M^{\eta}}))$, 
the space $\Hom_{K}(V_{\lambda}^{K},\stWhLgmg)$ 
is identified with 
\begin{align*}
\{\widetilde{\phi} \in 
C^{\infty}(A \rightarrow 
&\Hom_{M^{\eta}}(V_{\lambda}^{K}, V_{\gamma}^{M^{\eta}}))
\,|\,
\\
& z \cdot \widetilde{\phi} = \chi_{\Lambda}(z) \widetilde{\phi}, \ 
z \in Z(\lie{g});
\
\widetilde{\phi} \mbox{ grows moderately at infinity}\}.
\end{align*}
%
%
The space 
$\Hom_{M^{\eta}}
(V_{\lambda}^{K},V_{\gamma}^{M^{\eta}})$ is
isomorphic to 
$(V_{\lambda^{\ast}}^{K} \otimes 
V_{\gamma}^{M^{\eta}})^{M^{\eta}}$, 
the space of $M^{\eta}$-invariants in 
$V_{\lambda^{\ast}}^{K} \otimes V_{\gamma}^{M^{\eta}}$. 
By Remark~\ref{remark:GT restriction}, a basis of this space
is identified with the ``partial Gelfand-Tsetlin patterns'' 
\begin{align}
GT((\lambda/\gamma)^{\ast}) 
:= 
\{Q = &(\vect{q}_{r-3}, \vect{q}_{r-2}, \vect{q}_{r-1}) \,|\, 
\notag
\\
&\vect{q}_{r-3} = \gamma^{\ast}, \vect{q}_{r-1} = \lambda^{\ast}; \, 
\mbox{ satisfies Definition~\ref{definition:GT} (1)--(3)}\}. 
\notag
\end{align}
%
%
Let $V_{(\lambda/\gamma)^{\ast}}^{K/M^{\eta}}$ be the vector space
spanned by $GT((\lambda/\gamma)^{\ast})$ 
and let 
$C^{\infty}(A \rightarrow V_{(\lambda/\gamma)^{\ast}}^{K/M^{\eta}})$
be the space of $V_{(\lambda/\gamma)^{\ast}}^{K/M^{\eta}}$-valued 
$C^{\infty}$ functions on $A$. 
By the above discussion, 
$\Hom_{K}(V_{\lambda}^{K}, \stWhLgmg)$ is identified with 
a subspace of 
$C^{\infty}(A \rightarrow V_{(\lambda/\gamma)^{\ast}}^{K/M^{\eta}})$. 

%
Denote by $\Delta_{\lie{s}}$ the set of weights of the adjoint
representation $(\Ad, \lie{s})$ of $K$ on $\lie{s}$. 
For every $\alpha \in \Delta_{\lie{s}}$, 
$K$-type shift operators 
\begin{align*}
P_{\alpha} :\ 
& 
C^{\infty}(A \rightarrow V_{(\lambda/\gamma)^{\ast}}^{K/M^{\eta}})
&
&
\quad \rightarrow \quad 
&
&
C^{\infty}
(A \rightarrow V_{(\lambda+\alpha/\gamma)^{\ast}}^{K/M^{\eta}}),
\\
&
\hspace{15mm} \cup 
&
&
&
&
\hspace{15mm}
\cup 
\\
&
\Hom_{K}(V_{\lambda}^{K},\stWhLsmg)
&
&
\quad \rightarrow \quad 
&
&
\Hom_{K}(V_{\lambda+\alpha}^{K},\stWhLsmg) 
\end{align*}
realize the action of elements in $\lie{s}$ on $\stWhLgmg$ 
(see \cite[\S5]{T2}). 

For notational convenience, let $\epsilon_{-k} := \epsilon_{k}$ and 
$\epsilon_{0} = 0$. 
Firstly, consider the case $r=2n$. 
In this case, 
$\Delta_{\lie{s}} = \{\epsilon_{k}\, |\,
k=\pm 1,\dots,\pm n\}$. 
By Remark~\ref{remark:GT of dual}, 
$(\lambda + \epsilon_{k})^{\ast} 
= \lambda^{\ast} + \epsilon_{k}$ if $|k| < n$, 
and 
$(\lambda \pm \epsilon_{n})^{\ast} 
= \lambda^{\ast} \pm (-1)^{n} \epsilon_{n}$. 
Secondly, consider the case $r=2n+1$. 
In this case, 
$\Delta_{\lie{s}} = \{\epsilon_{k}\, |\,
k=0, \pm 1,\dots,\pm n\}$. 
By Remark~\ref{remark:GT of dual}, 
$(\lambda + \epsilon_{k})^{\ast} 
= \lambda + \epsilon_{k}$. 

For simplicity, 
denote by $P_{k}$ the operator 
$P_{\epsilon_{k}}$ for $k=0, \pm 1, \dots, \pm n$. 
When $r=2n$, the explicit forms of these operators are obtained in
\cite{T}. 
In the case $r=2n+1$, they are obtained just in the same way. 
\begin{proposition}\label{proposition:K-type shift}
Suppose $\phi(a) 
= \sum_{Q \in GT((\lambda/\gamma)^{\ast})} c(Q; a)\, Q
\in 
C^\infty(A \rightarrow V_{(\lambda/\gamma)^{\ast}}^{K/M^{\eta}})$. 
\begin{enumerate}
\item
When $r=2n$, the $K$-type shift operator $P_{k}$, 
$k=\pm 1,\dots,\pm(n-1)$ 
is given by the following formula: 
\begin{align}
P_{k} \phi(a) 
&= 
-\sum_{Q \in GT((\lambda/\gamma)^{\ast})} \hspace{-5mm}
a_{2n-1,k}(Q) 
(\partial_{a} - l_{2n-1,k}(Q)+n-1)\, c(Q; a)\, \sigma_{2n-1,k} Q 
\label{eq:P_k,2n}
\\
& \quad 
+\frac{\I \xi}{a} 
\sum_{0 \leq |j| \leq n-1} 
\sum_{\sigma_{2n-2,-j}Q \in GT((\lambda/\gamma)^{\ast})} 
\notag\\
& \qquad 
\times 
\frac{a_{2n-2,j}(\sigma_{2n-2,-j} Q)\, a_{2n-1,k}(Q)}
{l_{2n-2,j}(\sigma_{2n-2,-j}Q)-l_{2n-1,k}(Q)} 
c(\sigma_{2n-2,-j}Q; a)\, \sigma_{2n-1,k} Q. 
\notag
\end{align}
Here, $\partial_{a}$ is the differential operator $a (d/da)$. 
When $k=\pm n$, the ``$k$'' in the right hand side is replaced by 
$(-1)^{n} k$. 
\item
When $r=2n+1$, the $K$-type shift operator  $P_{k}$, 
$k=0, \pm 1,\dots,\pm n$ 
is given by the following formula: 
\begin{align}
P_{k} \phi(a) 
&= 
-\sum_{Q \in GT((\lambda/\gamma)^{\ast})}  \hspace{-5mm}
a_{2n,k}(Q) 
(\partial_{a} - l_{2n,k}(Q)+n)\, c(Q; a)\, \sigma_{2n,k} Q 
\label{eq:P_k,2n+1}
\\
& \quad 
+\frac{\I \xi}{a} 
\sum_{1 \leq |j| \leq n} 
\sum_{\sigma_{2n-1,-j}Q \in GT((\lambda/\gamma)^{\ast})} 
\notag\\
& \qquad 
\times 
\frac{a_{2n-1,j}(\sigma_{2n-1,-j} Q)\, a_{2n,k}(Q)}
{l_{2n-1,j}-l_{2n,k}(Q)} 
c(\sigma_{2n-1,-j}Q; a)\, \sigma_{2n,k} Q. 
\notag
\end{align}
\end{enumerate}
\end{proposition}

First, we complete the proof of 
Proposition~\ref{proposition:condition for sigma, integral,2n}(1) 
by using these operators. 
\begin{lemma}\label{lemma:pi_0,pi_1 sub}
The irreducible module  
$\pi_{0}$ (resp. $\pi_{1}$) is a submodule of $\stWhLgmg$ 
if and only if $\gamma$ 
satisfies \eqref{eq:condition for sigma,2n,0} 
(resp. \eqref{eq:condition for sigma,2n,1}). 
\end{lemma}
\begin{proof}
Let $\lambda$ be the minimal $K$-type of $\pi_{0}$ (resp. $\pi_{1}$). 
Suppose $\gamma$ satisfies \eqref{eq:condition for sigma,2n,0}
or \eqref{eq:condition for sigma,2n,1}. 
By \cite[Theorem~2.4]{Y}, an embedding of $\pi_{0}$ (resp. $\pi_{1}$)
into $C^{\infty}(G/V; \eta)_{K}$ is characterized by the system of
equations 
$P_{-1} \phi = 0, \dots, P_{-n} \phi = 0$ 
(resp. $P_{-1} \phi = 0, \dots, P_{-n+1} \phi = 0, P_{n} \phi = 0$) 
for $\phi \in 
C^\infty(A \rightarrow V_{(\lambda/\gamma)^{\ast}}^{K/M^{\eta}})$. 
This system of equations is solved in \cite{T} 
(though, notation is a little different). 

Let $Q_{0} = (\gamma^{\ast}, \vect{q}_{2n-2}, \lambda^{\ast})$ be a
vector in $GT((\lambda/\gamma)^{\ast})$ which satisfies 
$q_{2n-2,j}(Q_{0}) = \gamma_{j}^{\ast}$ ($j=1,2,\dots,n-2$) 
and $q_{2n-2,n-1} = |\gamma_{n-1}^{\ast}|$. 
Then the function $\phi$ characterizing the embedding 
of $\pi_{0}$ (resp. $\pi_{1}$) into $\stWhLgmg$ is determined by 
the coefficient function $c(Q_{0}; a)$. 
This function is a solution of the differential equation 
\begin{align*}
& 
\left(
\d_{a} + n - 1 
+ 
\sum_{p=1}^{n} \lambda_{p} 
- \sum_{p=1}^{n-2} \gamma_{p} - |\gamma_{n-1}| 
- 
(\sgn \gamma_{n-1}) \frac{\xi}{a} 
\right) 
c(Q_{0}; a) = 0 
\\
& 
(\mbox{resp. } 
\left(
\d_{a} + n - 1 
+ 
\sum_{p=1}^{n-1} \lambda_{p} - \lambda_{n}
- \sum_{p=1}^{n-2} \gamma_{p} - |\gamma_{n-1}| 
+ 
(\sgn \gamma_{n-1}) \frac{\xi}{a} 
\right) 
c(Q_{0}; a) = 0
). 
\end{align*}
Since we set $\xi > 0$, there exists non-zero moderate growth solution
if and only if $\gamma_{n-1} > 0$ (resp. $\gamma_{n-1} < 0$), 
i.e. $\gamma$ satisfies \eqref{eq:condition for sigma,2n,0} 
(resp. \eqref{eq:condition for sigma,2n,1}). 
\end{proof}

%
\begin{lemma}\label{lemma:condition for vanishing under shift}
Let $\phi$ be an element of 
$C^\infty(A \rightarrow V_{(\lambda/\gamma)^{\ast}}^{K/M^{\eta}})$. 
If $r$, $k$, $\lambda$ and $\gamma$ satisfies one of the following
conditions, then $P_{k} \phi = 0$ implies $\phi = 0$, i.e. 
$P_{k}$ is injective. 
\begin{enumerate}
\item
$r=2n$, $k\in\{2,\dots,n-1\}$, 
$\gamma_{k-1} > \lambda_{k}$. 
\item
$r=2n$, $k\in\{-2,\dots,-(n-1)\}$, 
$\gamma_{|k|-1} < \lambda_{|k|}$. 
\item
$r=2n$, $\lambda_{n}^{\ast} = (-1)^{n} \lambda_{n} > 0$, 
$\lambda_{n}^{\ast} > 
|\gamma_{n-1}^{\ast}| 
= 
|(-1)^{n-1} \gamma_{n-1}|$, 
$k=-(-1)^{n} n$. 
\item
$r=2n$, $\lambda_{n}^{\ast} = (-1)^{n} \lambda_{n} < 0$, 
$-\lambda_{n}^{\ast} > 
|\gamma_{n-1}|
= 
|(-1)^{n-1} \gamma_{n-1}|$, 
$k=(-1)^{n} n$. 
\item
$r=2n+1$, 
$k\in\{2,\dots,n\}$, 
$\gamma_{k-1} > \lambda_{k}$. 
\item
$r=2n+1$, $k\in\{-2,\dots,-n\}$, 
$\gamma_{|k|-1} < \lambda_{|k|}$. 
\end{enumerate}
\end{lemma}
\begin{proof}
Since the proofs of these are analogous, we show only (1). 
We will show $c(Q; a) = 0$ by induction 
on $\lambda_{k}^{\ast}-q_{2n-2,k}(Q)$. 

Let $Q_{0}$ be an element of $GT((\lambda/\gamma)^{\ast})$ which
satisfies $q_{2n-2,k}(Q_{0}) = \lambda_{k}^{\ast} = \lambda_{k}$, 
and let 
$Q_{1} := \sigma_{2n-2,k} Q_{0}$. 
Then $Q_{1}$ is not in $GT((\lambda/\gamma)^{\ast})$, 
but $\sigma_{2n-2,-k} Q_{1} = Q_{0} \in GT((\lambda/\gamma)^{\ast})$
and 
$\sigma_{2n-1,k} Q_{1} \in GT((\lambda+\epsilon_{k}/\gamma)^{\ast})$, 
because 
$\gamma_{k-1}^{\ast} = \gamma_{k-1}> \lambda_{k} = \lambda_{k}^{\ast}$
implies that $\sigma_{2n-1,k} Q_{1}$ satisfies the conditions in
  Definition~\ref{definition:GT}~(3): 
$q_{2n-2,k-1}(Q_{1}) \geq 
\gamma_{k-1}^{\ast} \geq \lambda_{k}^{\ast}+1 
= q_{2n-1,k}(\sigma_{2n-1,k} Q_{1}) 
= q_{2n-2,k}(\sigma_{2n-2,k} Q_{1})$. 
Therefore, the term $\sigma_{2n-2,k} Q_{1}$ appears in \eqref{eq:P_k,2n}, 
and its coefficient in \eqref{eq:P_k,2n} is 
\begin{align}
&\frac{a_{2n-2,k}(\sigma_{2n-2,-k} Q_{1})\, a_{2n-1,k}(Q_{1})}
{l_{2n-2,k}(\sigma_{2n-2,-k} Q_{1})-l_{2n-1,k}(Q_{1})} 
c(\sigma_{2n-2,-k} Q_{1}; a)
\notag\\
&\qquad = 
\frac{a_{2n-2,k}(Q_{0}) \, a_{2n-1,k}(\sigma_{2n-2,k} Q_{0})}
{l_{2n-2,k}(Q_{0}) - l_{2n-1,k}(Q_{0})} 
c(Q_{0}; a). 
\label{eq:coeff of c(Q_1;a)}
\end{align}
In general, 
\begin{align*}
\frac{a_{2n-2,j}(Q)\, a_{2n-1,k}(\sigma_{2n-2,j} Q)}
{l_{2n-2,j}(Q) - l_{2n-1,k}(Q)}
=
\frac{a_{2n-2,j}(\sigma_{2n-1,k} Q)\, a_{2n-1,k}(Q)}
{l_{2n-2,j}(Q) - l_{2n-1,k}(Q) - 1}. 
\end{align*}
Here, we used the definition of $a_{i,j}(Q)$. 
Therefore, \eqref{eq:coeff of c(Q_1;a)} is 
\[
\frac{a_{2n-2,k}(\sigma_{2n-1,k} Q_{0}) 
a_{2n-1,k}(Q_{0})}
{(\lambda_{k}+n-k)-(\lambda_{k}+n-k)-1} 
c(Q_{0}; a). 
\]
It is easy to check that 
$a_{2n-2,k}(\sigma_{2n-1,k} Q_{0})\, 
a_{2n-1,k}(Q_{0})$ is not zero. 
So if $P_{k} \phi = 0$, then $c(Q_{0}; a)=0$. 
We have shown that $c(Q; a)$ is zero for those $Q$ which satisfy 
$q_{2n-2,k}(Q)=\lambda_{k}^{\ast}$. 

Assume that $c(Q; a) = 0$ is proved for those $Q$'s which satisfy 
$\lambda_{k}^{\ast}-q_{2n-2,k}(Q)=p$. 
Let $Q_{2}$ be an element of $GT((\lambda/\gamma)^{\ast})$ which
satisfies $\lambda_{k}^{\ast}- q_{2n-2,k}(Q_{2}) = p+1$. 
Set $Q_{3} := \sigma_{2n-2,k}^{+}Q_{2}$. 
This $Q_{3}$ is an element of $GT((\lambda/\gamma)^{\ast})$ 
and it satisfies $\lambda_{k}^{\ast}-q_{2n-2,k}(Q_{3})=p$ 
and 
$\lambda_{k}^{\ast}- q_{2n-2,k}(\sigma_{2n-2,-i}Q_{3}) = p$ 
if $i\not=k$. 
Then by the hypothesis of induction, 
$c(Q_{3}; a) = 0$ and $c(\sigma_{2n-2,-i}Q_{3}; a) = 0$ 
for $i\not=k$. 
Consider the right hand side of \eqref{eq:P_k,2n} for $Q = Q_{3}$. 
The terms other than $c(Q_{2}; a)=c(\sigma_{2n-2,-k}Q_{3}; a)$ are
zero. 
We can easily see that its coefficient 
\[
\frac{a_{2n-2,k}(\sigma_{2n-2,-k} Q_{3})\, a_{2n-1,k}(Q_{3})}
{l_{2n-2,k}(\sigma_{2n-2,-k}Q_{3})-l_{2n-1,k}(Q_{3})}
\]
is not zero. 
Therefore, if $P_{k}\phi(a) = 0$, then $c(Q_{2}; a)=0$. 
This completes the proof. 
\end{proof}

Choose a Cartan subalgebra 
$\lie{h} 
:= 
\bigoplus_{i=1}^{\lfloor(r+1)/2\rfloor}
\C A_{r-2i+3,r-2i+2}$ of $\lie{g}$, and let 
$\gamma: Z(\lie{so}_{r+1}) 
\to U(\lie{h})^{W(\lie{g},\lie{h})}$ be the Harish-Chandra isomorphism. 
The following theorem is proved in \cite{T3}. 
\begin{theorem}
[{\cite[Theorem~1.1, Lemma~3.2 and Proposition~5.3]{T3}}]
\label{theorem:central elements in Ug}
For $u \in \C$, let $C_{r+1}(u)$ be the element in
$Z(\lie{so}_{r+1})$ which satisfies 
%
\begin{equation}\label{eq:image of HC map}
\gamma(C_{r+1}(u)) 
= 
\prod_{i=1}^{\lfloor (r+1)/2 \rfloor}
(u^{2} + A_{r-2i+3,r-2i+2}^{2}). 
\end{equation}
%
%
When $r=2n+1$, 
let $\mathbb{PF}_{2n+2}$ be the element in $Z(\lie{so}_{2n+2})$ which
satisfies 
\[
\gamma(\mathbb{PF}_{2n+2}) 
= 
(-\I)^{n+1} A_{2,1} A_{4,2} \cdots A_{2n+2,2n+1}.
\]
For $\tau_{\lambda} \in \widehat{K}$, define 
\begin{equation}\label{eq:u_k}
u_{k} := 
\begin{cases}
l_{2n-1,k} + 1/2 \ \mbox{ when }\ r=2n \ \mbox{ and } \
|k| < n, 
\\
l_{2n-1,(-1)^{n} k} + 1/2 \ \mbox{ when }\ r=2n \ \mbox{ and } \
k= \pm n, 
\\
l_{2n,k} \ \mbox{ when }\ r=2n+1. 
\end{cases}
\end{equation}
%
\begin{enumerate}
\item
For $k = \pm 1, \dots, \pm \lfloor r/2 \rfloor$, 
there exists a non-zero constant $d_{\lambda,k}$ such that 
\begin{equation}\label{eq:P_-k P_k} 
P_{-k} \circ P_{k} \phi
= d_{\lambda,k}\, L(C_{r+1}(u_{k})) \phi, 
\qquad  
\phi \in 
C^{\infty}(A \rightarrow V_{(\lambda/\gamma)^{\ast}}^{K/M^{\eta}}). 
\end{equation}
\item
When $r=2n+1$, 
there exists a non-zero constant $d_{\lambda}$ such that 
\begin{equation}\label{eq:P_0 and Pfaffian}
P_{0} \phi = d_{\lambda}\, L(\mathbb{PF}_{2n+2})\, \phi, 
\qquad 
\phi \in 
C^{\infty}(A \rightarrow V_{(\lambda/\gamma)^{\ast}}^{K/M^{\eta}}). 
\end{equation}
\end{enumerate}
\end{theorem}

\section{Determination of composition series}
\label{section:determination of composition series} 

In this section, we determine the socle filtration of $\stWhLgmg$.

\begin{lemma}\label{lemma:zero shifts}
Let $V_{1}$ and $V_{2}$ are irreducible factors in $\stWhLgmg$
which satisfy one of the following conditions: 
\begin{enumerate}
\item\label{lemma:zero shifts(10)}
$V_{1} \simeq V_{2}$, 
but they are different irreducible factors. 
\item \label{lemma:zero shifts(1)}
$r=2n$, $\gamma$ satisfies \eqref{eq:condition for sigma,2n,0}, 
$V_{1} \simeq \pi_{0}$, 
$V_{2} \simeq \overline{\pi}_{0,n}$. 
\item \label{lemma:zero shifts(2)}
$r=2n$, $\gamma$ satisfies \eqref{eq:condition for sigma,2n,0}, 
$V_{1} \simeq \overline{\pi}_{0,n}$, 
$V_{2} \simeq \pi_{1}$. 
\item \label{lemma:zero shifts(3)}
$r=2n$, $\gamma$ satisfies \eqref{eq:condition for sigma,2n,1}, 
$V_{1} \simeq \pi_{1}$, 
$V_{2} \simeq \overline{\pi}_{0,n}$. 
\item \label{lemma:zero shifts(4)}
$r=2n$, $\gamma$ satisfies \eqref{eq:condition for sigma,2n,1}, 
$V_{1} \simeq \overline{\pi}_{0,n}$, 
$V_{2} \simeq \pi_{0}$. 
\item \label{lemma:zero shifts(5)}
$r=2n$, $i \in \{2, \dots, n-1\}$, 
$\gamma$ satisfies \eqref{eq:condition for sigma,2n,0i}, 
$V_{1} \simeq \overline{\pi}_{0,i}$, 
$V_{2} \simeq \overline{\pi}_{0,i-1}$ or $\overline{\pi}_{0,i+1}$. 
\item \label{lemma:zero shifts(6)}
$r=2n$, 
$\gamma$ satisfies \eqref{eq:condition for sigma,2n,0i} for $i=n$, 
$V_{1} \simeq \overline{\pi}_{0,n}$, 
$V_{2} \simeq \overline{\pi}_{0,n-1}$, $\pi_{0}$ or $\pi_{1}$. 
\item \label{lemma:zero shifts(8)}
$r=2n+1$, $i \in \{2, \dots, n\}$, 
$\gamma$ satisfies \eqref{eq:condition for sigma,2n+1,0i}, 
$V_{1} \simeq \overline{\pi}_{0,i}$, 
$V_{2} \simeq \overline{\pi}_{0,i-1}$ or $\overline{\pi}_{0,i+1}$. 
\item \label{lemma:zero shifts(9)}
$r=2n+1$, 
$\gamma$ satisfies \eqref{eq:condition for sigma,2n+1,0i} for $i=n+1$, 
$V_{1} \simeq \overline{\pi}_{0,n+1}$, 
$V_{2} \simeq \overline{\pi}_{0,n}$. 
\end{enumerate}
Then there is no non-zero $\lie{g}$-action in
$\stWhLgmg$ which sends an element of $V_{1}$ to $V_{2}$. 
\end{lemma}

\begin{proof}
If there is a $\lie{g}$-action sending an element of 
$V_{1}$ to $V_{2}$, 
then the $K$-spectra $\widehat{K}(V_{1})$ and 
$\widehat{K}(V_{2})$ should be adjacent, i.e. there
should be $K$-types 
$\tau_{\lambda} \in \widehat{K}(V_{1})$ 
and $\tau_{\lambda'} \in \widehat{K}(V_{2})$ such
that $\lambda'-\lambda$ is a weight of $\lie{s}$. 
We set $\epsilon_{k} = \lambda' - \lambda$. 

First, we show \eqref{lemma:zero shifts(10)}. 
Recall the discussion in \S~\ref{subsection:shift operators}. 
The non-zero $\lie{s}$-action which sends an element of
$V_{\lambda}^{K} \subset V_{1}$ 
to 
$V_{\lambda'}^{K} \subset V_{2}$ is realized by the shift operator
$P_{k} \phi \not= 0$, 
where 
$\phi$ is a non-zero element in 
$C^{\infty}(A \rightarrow V_{(\lambda/\gamma)^{\ast}}^{K/M^{\eta}})$
corresponding to an element in 
$\Hom_{K}(V_{\lambda}^{K}, \stWhLgmg)$. 

If $r=2n+1$ and $k=0$, 
then $P_{0} \phi = d_{\lambda}\, L(\mathbb{PF}_{2n+2}) \phi$
by \eqref{eq:P_0 and Pfaffian}. 
Since $\mathbb{PF}_{2n+2}$ is a central element, 
$P_{0} \phi$ is a constant multiple of $\phi$. 
So it realizes the $K$-type $V_{\lambda}^{K} \subset V_{1}$ or it is a
zero element. 
In either case, $P_{0} \phi$ does not realize the
$K$-type $V_{\lambda'}^{K} \subset V_{2} \not= V_{1}$. 

Suppose that $k$ is not zero. 
Consider the shift $P_{-k} \circ P_{k} \phi$. 
Lemma~\ref{theorem:central elements in Ug}~(1) asserts 
that 
$P_{-k} \circ P_{k} \phi
= 
d_{\lambda,k} \, L(C_{r+1}(u_{k})) \phi$. 
Since $C_{r+1}(u_{k}) \in Z(\lie{g})$, it acts by the scalar 
$\chi_{\Lambda}(C_{r+1}(u_{k}))$. 
By \eqref{eq:image of HC map}, 
the image of Harish-Chandra map of $C_{r+1}(u_{k})$ is 
\[
\chi_{\Lambda}(C_{r+1}(u_{k})) 
= 
\prod_{i=1}^{\lfloor (r+1)/2 \rfloor} 
(u_{k}^{2} - \Lambda_{i}^{2}). 
\]
By the definition \eqref{eq:u_k} of $u_{k}$, 
the scalar $\chi_{\Lambda}(C_{r+1}(u_{k}))$ is zero if and only if one
of the following conditions is satisfied: 
\begin{align}
& r = 2n, \ k > 0, \ 
\lambda_{k} = \Lambda_{i} - n + k - 1/2 \ 
(\exists i\in \{1,2,\dots,n\}), 
\label{eq:central character is zero(1)}
\\
& r = 2n, \ k < 0, \ 
\lambda_{|k|} = \Lambda_{i} - n + |k| + 1/2 \ 
(\exists i\in \{1,2,\dots,n\}), 
\label{eq:central character is zero(2)}
\\
& r = 2n, \ k = n, \ 
\lambda_{n} = \pm \Lambda_{i}-1/2 \ 
(\exists i\in \{1,2,\dots,n\}), 
\\
& r = 2n, \ k = -n, \ 
\lambda_{n} = \pm \Lambda_{i}+1/2 \ 
(\exists i \in \{1,2,\dots,n\}), 
\\
& r = 2n+1, \ k > 0, \ 
\lambda_{k} = \Lambda_{i} - n + k - 1 \ 
(\exists i \in \{1,2,\dots,n\})\ \mbox{ or } |\Lambda_{n+1}|,  
\label{eq:central character is zero(3)}
\\
& r = 2n+1, \ k < 0, \ 
\lambda_{|k|} = \Lambda_{i} - n + |k|  \ 
(\exists i\in \{1,2,\dots,n\}) \mbox{ or } |\Lambda_{n+1}|+1. 
\label{eq:central character is zero(4)}
\end{align}

Recall the $K$-spectra of irreducible $\brgK$-modules 
(Theorems~\ref{theorem:K-spectra,2n}, \ref{theorem:K-spectra,2n+1}). 
If $\lambda$ and $k$ satisfy one of 
\eqref{eq:central character is zero(1)} 
-- 
\eqref{eq:central character is zero(4)}, 
then the $K$-type $V_{\lambda'}^{K} = V_{\lambda+\epsilon_{k}}^{K}$
is not a $K$-type of $V_{1} \simeq V_{2}$. 
Therefore, if $V_{1} \simeq V_{2}$ and $P_{k} \phi \not= 0$, 
then $P_{-k} \circ P_{k} \phi$ is a non-zero multiple of $\phi$. 
This is impossible since $V_{1}$ and $V_{2}$ are different
irreducible subquotients and there exist non-zero $\lie{s}$-actions
sending an element in $V_{1}$ to $V_{2}$ and vice versa. 
Therefore, \eqref{lemma:zero shifts(10)} is proved.

Let us show the case \eqref{lemma:zero shifts(9)}. 
By Theorem~\ref{theorem:K-spectra,2n+1}, 
two $K$-types $\lambda \in \widehat{K}(V_{1})$ and 
$\lambda' \in \widehat{K}(V_{2})$ are adjacent if and
only if $\lambda_{n} = \Lambda_{n}$,
$\lambda_{n}'=\Lambda_{n}-1$ 
and $\lambda_{p}=\lambda_{p}'$ for $p=1,\dots,n-1$. 
In this case, the shift operator sending $V_{\lambda}^{K}$ to
$V_{\lambda'}^{K}$ is $P_{-n}$. 
By \eqref{eq:central character is zero(4)}, 
$P_{n} \circ P_{-n} \phi = 0$. 
We know that $\gamma_{n-1} \geq \Lambda_{n}$ 
since $\gamma$ satisfies \eqref{eq:condition for sigma,2n+1,0i} for
$i=n+1$. 
On the other hand, 
$\lambda_{n}' = \Lambda_{n}-1$. 
Then the condition of 
Lemma~\ref{lemma:condition for vanishing under shift}~(5) 
for $k=n$ (and $\lambda$ is replaced by $\lambda'$) is satisfied. 
So $P_{n}$ is injective. 
It follows that $P_{-n} \phi = 0$, so that there is no non-zero
$\lie{g}$-action sending $V_{1}$ to $V_{2}$. 
This proves \eqref{lemma:zero shifts(9)}. 

We can show 
\eqref{lemma:zero shifts(1)}--\eqref{lemma:zero shifts(8)} 
analogously. 
\end{proof}

\begin{corollary}\label{corollary:pi_ij is multiplicity free}
Suppose that the unique irreducible submodule of
$\stWhLgmg$ is $\pi$. 
Then the multiplicity of $\pi$ in $\stWhLgmg$ is one. 
\end{corollary}
\begin{proof}
Assume that there exists a composition factor $V_{1}$ 
which is isomorphic to $\pi$ but is not the unique irreducible
submodule. 
Then there exists an irreducible subquotient $V_{2}$ of $\stWhLgmg$
such that $V_{1} \to V_{2}$. 
By Proposition~\ref{proposition:composition factors-1}, 
we know that the pair $(V_{1}, V_{2})$ satisfies one of the conditions
in Lemma~\ref{lemma:zero shifts}. 
But this lemma tells us that there is no non-zero $\lie{g}$-action
sending an element of $V_{1}$ to $V_{2}$. 
Therefore, $V_{1} \not\rightarrow V_{2}$. 
\end{proof}

In order to determine the second and higher floors of $\stWhLgmg$, 
the next theorem is very useful. 

\begin{theorem}[{\cite[Theorem~9.5.1]{V2}}]
\label{theorem:Ext}
In the setting of this paper, 
suppose that irreducible $\brgK$-modules $X$ and $Y$ are not isomorphic. 
Then $\Ext_{\gK}^{1}(X,Y) \not= 0$ only if 
$\ell(X) - \ell(Y) \equiv 1 \mod 2$. 
\end{theorem}

This theorem imposes a parity structure on $\stWhLgmg$. 
We know that the socle of it consists of a single irreducible module. 
If the length of the module in the socle is even (resp. odd), 
then by this theorem, the lengths of the irreducible factors in the
second floor are odd (resp. even), the third floor even (resp. odd), 
and so on. 

\vspace{2mm}

\noindent
{\it 
Proofs of Theorems~\ref{theorem:main, 2n}, \ref{theorem:main, 2n+1}.}
The statements (1) in these theorems are proved in 
Propositions~\ref{proposition:condition for sigma, integral,2n}, 
\ref{proposition:condition for sigma, integral,2n+1}. 

Let us consider the case when $r=2n$ and $\gamma$ satisfies 
\eqref{eq:condition for sigma,2n,0}. 
In this case, 
Proposition~\ref{proposition:condition for sigma, integral,2n} says 
that $\pi_{0}$ is the unique simple submodule of
$\stWhLgmg$, 
and Proposition~\ref{proposition:composition factors-1}~(1) says that
only $\pi_{0}, \pi_{1}$ or $\overline{\pi}_{0,n}$ can be a 
composition factor of $\stWhLgmg$. 
By Corollary~\ref{corollary:pi_ij is multiplicity free}, 
the multiplicity of $\pi_{0}$ in $\stWhLgmg$ is just one. 
By Theorem~\ref{theorem:classification of irr modules, 2n}, 
$\ell(\pi_{0}) = \ell(\pi_{1}) = 0$ 
and $\ell(\overline{\pi}_{0,n}) = 1$. 
Therefore, Theorem~\ref{theorem:Ext} implies that 
the second floor of $\stWhLgmg$ is a multiple of 
$\overline{\pi}_{0,n}$, and the third floor of it is a multiple of
$\pi_{1}$, and so on. 
We know from Proposition~\ref{proposition:at least one} 
that the multiplicities of $\overline{\pi}_{0,n}$ and $\pi_{1}$ 
in $\stWhLgmg$ are at least one. 
Therefore, the multiplicity of $\overline{\pi}_{0,n}$
(resp. $\pi_{1}$) in the second (resp. third) floor is at least one. 
We can show that the multiplicity of $\overline{\pi}_{0,n}$ 
(resp. $\pi_{1}$) in the second (resp. third) floor is just one, 
in the same way as the proof of \cite[Lemma~5.14]{T2}. 

Assume that there exists non-zero fourth floor in $\stWhLgmg$. 
Then there exists a $\lie{g}$ action which sends an 
element in the fourth floor to the third floor. 
But the fourth floor is a multiple of $\overline{\pi}_{0,n}$ and the
third floor is isomorphic to $\pi_{1}$. 
This contradicts Lemma~\ref{lemma:zero shifts} (3). 
Therefore, there is no fourth floor in $\stWhLgmg$.

The proof of the case when $\gamma$ satisfies 
\eqref{eq:condition for sigma,2n,1} is just the same as above. 

The proofs of Theorem~\ref{theorem:main, 2n}\, (3), (4) and 
Theorem~\ref{theorem:main, 2n+1}, (2), (3) are almost the same
as above (easier). 
For example, suppose $r=2n$ and $\gamma$ satisfies 
\eqref{eq:condition for sigma,2n,0i}, $i \in \{2,\dots,n-1\}$. 
Then $\overline{\pi}_{0,i}$ is the unique simple submodule of
$\stWhLgmg$. 
A composition factor of $\stWhLgmg$ is isomorphic to
$\overline{\pi}_{0,i}$, $\overline{\pi}_{0,i-1}$ or 
$\overline{\pi}_{0,i+1}$. 
Since $\ell(\overline{\pi}_{0,k}) = n-k+1$, $k=i-1, i, i+1$, 
and since the multiplicity of $\overline{\pi}_{0,i}$ 
in $\stWhLgmg$ is just one, 
the second floor of $\stWhLgmg$ is a direct sum of multiples of 
$\overline{\pi}_{0,i-1}$ and $\overline{\pi}_{0,i+1}$, 
and there is no higher floor. 
The multiplicity freeness of the factors in the second floor is proved 
in the same way as the proof of \cite[Lemma~5.14]{T2}. 
\qed

\end{document}